\documentclass[11pt]{article} 
\usepackage{amssymb, latexsym}
\newtheorem{defin}{}
\newtheorem{saetze}[defin]{}
\newtheorem{conjec}[defin]{}
\newtheorem{lemmas}[defin]{}
\newtheorem{folger}[defin]{}
\newtheorem{bemerk}[defin]{}

\newenvironment{lemma}    {\begin{lemmas}\it {\bf Lemma:}}{\end{lemmas}}

\newenvironment{remark}   {\begin{bemerk}\rm {\it Remark:}}{\end{bemerk}}
\newenvironment{proof}    {{\it Proof}:}{{\hfill \fillbox \bigskip}}

\newcommand{\fillbox}{\mbox{$\bullet$}}

\newcommand{\ol}{\overline}

\newcommand{\Z}{\mathbb Z}

\newenvironment{items}{\begin{list}{$\alph{item})$}
{\labelwidth18pt \leftmargin18pt \topsep3pt \itemsep1pt \parsep0pt}}
{\end{list}}

\newcommand{\bulit}{\item[$\bullet$]}

\begin{document}

\title{Computing the order and the index of \\
       a subgroup in a polycyclic group}
\author{Bettina Eick}
\date{\today}
\maketitle

\begin{abstract}
This contains a new version of the so-called 'non-commutative Gauss' 
algorithm for polycyclic groups. Its results allow to read off the 
order and the index of a subgroup in an (possibly infinite) polycyclic
group.
\end{abstract}

\section{Introduction}

Practical algorithms to compute with finite polycyclic groups have been 
described by Laue, Neub\"user \& Schoenwaelder in \cite{LNS84}. A main 
basis for most of these algorithms is the so-called 'non-commutative 
Gauss' algorithm. Given a finite polycyclic group $G$ and a finite set
of generators for a subgroup $U$, this computes a so-called induced 
polycyclic generating sequence for $U$. In turn, this allows to read 
off $|U|$ and $[G:U]$ and forms a basis for many other algorithms.

Eick \cite{Eic01} introduced various practical algorithms to compute with 
possibly infinite polycyclic groups. These also included an extended version
of the 'non-commutative Gauss' algorithm. The proof of this extended version
has gaps. It is the main aim here to cover this open problem and to 
introduce a practical and reliable version of the 'non-commutative Gauss'
algorithm for possibly infinite polycyclic groups.

A GAP implementation of the code described here is available online,
see \cite{Eic21}.

\section{Preliminaries}

We first introduce the setting of the algorithm. We assume that the 
group $G$ is given by a consistent polycyclic presentation; that is, 
it has generators $g_1, \ldots, g_n$ and non-negative integers 
$r_1, \ldots, r_n$ and relations of the form
\begin{eqnarray*}
g_i g_j &=& g_j g_{j+1}^{a_{i,j,j+1}} \cdots g_n^{a_{i,j,n}} 
 \mbox{ for } 1 \leq j < i \leq n\\
g_i g_j^{-1} &=& g_j^{-1} g_{j+1}^{b_{i,j,j+1}} \cdots g_n^{b_{i,j,n}} 
 \mbox{ for } 1 \leq j < i \leq n \mbox{ with } r_i = 0 \\
g_i^{r_i} &=& g_{i+1}^{c_{i,i+1}} \cdots g_n^{c_{i,n}}
 \mbox{ for } 1 \leq i \leq n \mbox{ with } r_i > 0,
\end{eqnarray*}
where $a_{i,j,k}, b_{i,j,k}, c_{i,k}$ are integers that are contained in
$\{0, \ldots, r_k-1\}$ if $r_k > 0$. Additionally, it is required that
for each element $g$ of $G$ there exists a unique $(e_1, \ldots, e_n) \in
\Z^n$ with $0 \leq e_k < r_k$ if $r_k>0$ so that 
\[ g = g_1^{e_1} \cdots g_n^{e_n}.\]

We call $g_1^{e_1} \cdots g_n^{e_n}$ the {\em normal form} of the element
$g$. It can be computed readily for any arbitrary word in the generators
by iteratedly applying the relations of the group.

Let $G_i = \langle g_i, \ldots, g_n \rangle$ for $1 \leq i \leq n$. The 
relations of $G$ imply that $G_{i+1} \unlhd G_i$ for $1 \leq i \leq n-1$
and $G_i/G_{i+1}$ is cyclic. The consistency of the presentation implies
that $[G_i : G_{i+1}] = r_i$ if $r_i > 0$ and $[G_i: G_{i+1}]=\infty$
if $r_i = 0$.

We introduce some further notation for elements in a group $G$ given by
a consistent polycyclic presentation. Let $g = g_1^{e_1} \cdots g_n^{e_n}$
be a normal form and assume that $g \neq 1$. Then
\begin{items}
\bulit
the {\em depth} of $g$ is $d$ if $e_1 = \ldots = e_{d-1} = 0$ and 
$e_d \neq 0$. Write $d(g)$.
\bulit
the {\em leading exponent} of $g$ is $e_d$. Write $l(g)$.
\bulit
the {\em relative order} of $g$ is $|g G_{d+1}|$. Write $r(g)$.
\end{items}

If $g = 1$, then we say that $g$ has depth $n+1$ and leading exponent or
relative order of $g$ do not exist. If $g \neq 1$, then $r(g) = \infty$ 
if $r_d = 0$ and $r(g) \mid r_d$ otherwise.

\section{Subgroups, Igs and Cgs}

Let $G$ be given by a consistent polycyclic presentation. A generating
set $u_1, \ldots, u_m$ of a subgroup $U$ is an {\em igs} if the series
$U_i = \langle u_i, \ldots, u_m \rangle$ with $1 \leq i \leq m$ coincides
with the series $G_i \cap U$ for $1 \leq i \leq n$ where duplicates have
been removed. The following has been proved in \cite{Eic01}.

\begin{lemma}
\label{igs}
Let $u_1, \ldots, u_m$ be a generating set for $U$. Then $u_1, \ldots, u_m$ 
is an igs for $U$ if and only if
\begin{items}
\bulit
$u_i^{u_j} \in U_{j+1}$ for $1 \leq j < i \leq m$, 
\bulit
$u_i^{r(u_i)} \in U_{i+1}$ for $1 \leq i \leq m$ with $r(u_i) > 0$, 
\bulit
$d(u_1) < d(u_2) < \ldots < d(u_m)$.
\end{items}
\end{lemma}

\begin{proof}
We include a proof for completeness. \\
(1) First assume that $u_1, \ldots, u_m$ is an igs. Choose $j$ maximal with
$U_i = G_j \cap U$. Then $U_{i+1} = G_{j+1} \cap U$ and thus $U_{i+1}
\unlhd U_i$ with $U_i/U_{i+1}$ of order $r(u_i)$. Thus all three items
follow. \\
(2) Now assume that the three items are satisfied. By item (a) it follows
that $U_{i+1} \unlhd U_i$ and by construction and item (b) the quotient
$U_i/U_{i+1}$ is cyclic of order $r(u_i)$ if $r(u_i)>0$ and cyclic of order 
$\infty$ if $r(u_i) = 0$. Induction now yields the desired result.
\end{proof}

\section{Computing an igs}

We assume that generators $h_1, \ldots, h_l$ for a subgroup $U$ of $G$
are given. Our aim is to determine an igs for $U$. 

\subsection{Normalisations of elements}

Let $g \in G$, $g \neq 1$, with depth $d$, leading exponent $a$ and 
relative order $r_d$. If $r_d = 0$, then let $e = sign(a)$ and call 
$g^e$ the {\em normalisation} of $g$. If $r_d > 0$, then write $a = xy$ 
with $x = gcd(a,m)$ and $y = a/x$. Note that $z = y^{-1} \bmod m$ exists
and call $g^z$ the {\em normalization} of $g$. 

\begin{remark}
\label{norm}
Let $g \in G$, $g \neq 1$, with depth $d$ and normalisation $h$.
\begin{items}
\item[\rm (a)]
$d(h) = d(g)$
\item[\rm (b)]
$l(h) \mid l(g)$.
\item[\rm (c)] 
$\langle g, G_{d+1} \rangle = \langle h, G_{d+1} \rangle$.
\end{items}
\end{remark}

Remark \ref{norm} indicates why the normalisation of an element is of
interest: in the cyclic group $G_d/G_{d+1}$ it yields the unique generator 
of $\langle g G_{d+1} \rangle$ with smallest leading exponent.

\subsection{Partial Igs}

A {\em partial igs} is a list of length $n$ (the number of generators of 
the parent group $G$) whose $i$th entry is either empty or a normalised
element $g$ in $G$ of depth $i$.

The following function takes a partial igs $I$ and an arbitrary element 
$g \in G$ and determines a new partial igs $J$ so that $\langle J \rangle 
= \langle I, g \rangle$. 
\bigskip

{\bf AddGenToPIgs}($I$, $g$):
\begin{items}
\item[(1)] 
Initialise $L$ as the list with a single entry $g$.
\item[(2)] 
While $L$ is not empty do:
\begin{items}
\item[(a)]
Take an element $h$ from $L$ and eliminate it in $L$.
\item[(b)]
Let $d = d(h)$. If $d > n$ then go back to (2).
\item[(c)]
If $I[d]$ is empty then: 
\begin{items}
\item[(i)]
Insert the normalisation of $h$ at position $d$ in $I$.
\item[(ii)]
If $r_d > 0$ then add $h \cdot I[d]^{-l(h)/l(I[d])}$ to $L$.
\end{items}
\item[(d)]
If $I[d]$ is not empty then:
\begin{items}
\item[(i)]
Let $k = I[d]$ and $b = l(k)$ and $a = l(h)$. 
\item[(ii)]
Let $e = gcd(a,b) = ua + vb$ and $w = h^u k^v$.
\item[(iii)]
Insert the normalisation of $w$ at position $d$ in $I$.
\item[(iv)]
Add $h \cdot I[d]^{-l(h)/l(I[d])}$ to $L$.
\item[(v)]
Add $k \cdot I[d]^{-l(k)/l(I[d])}$ to $L$.
\end{items}
\end{items}
\end{items}
\bigskip

First, note that in Steps (2d)(iv) and (2d)(v), the quotients $l(h)/l(I[d])$
and $l(k)/l(I[d])$ are integers, since $l(I[d]) \mid l(w) = e = gcd(a,b)$
and $a = l(h)$ and $b=l(k)$. Hence these Steps yield elements of $G$ that
are added to $L$. 

Second, in the Steps (2c)(ii),  (2d)(iv) and (2d)(v) there are elements of $G$
added to $L$. All of these elements have depth greater than $d$. This
implies that the algorithm terminates eventually.

\begin{lemma}
Let $J = AddGenToPIgs(I, g)$. Then $J$ is a partial igs
satisfying $\langle J \rangle = \langle I, g \rangle$.
\end{lemma}

\begin{proof}
$J$ is a partial igs, since we only add normalised elements at the places 
associated with their depth. It remains to prove $\langle J \rangle = 
\langle I, g \rangle$.
\medskip

$\subseteq$: 
Each element that is inserted into $I$ during the algorithm is a product
of elements of $I \cup \{g\}$. Hence $J \subseteq \langle I, g \rangle$
and this part follows. 
\medskip

$\supseteq$: 
We show that $\langle L, I \rangle$ does not change in the course of the
algorithm. Since $I \cup \{g\} = I \cup L$ to begin with and $J = I \cup
L$ at the end, this yields the desired result. We consider the changes made
to $I$ and $L$ in the course of the algorithm. In Step (2a) we take an 
element $h$ from $L$. There are several cases: \\
(Case 1): $I[d]$ is empty and $r_d = 0$. Then we add $h$ or $h^{-1}$ to 
$I$ and the result follows. \\
(Case 2): $I[d]$ is empty and $r_d > 0$. Then we add the normalization
$h^z$ to $I$ and $h \cdot (h^z)^{-q}$ to $L$ for some $q$. Hence $h 
= h \cdot (h^z)^{-q} \cdot (h^z)^q$ can be obtained from $L$ and $I$ 
and the result follows. \\
(Case 3): $I[d] = k$. Then we add the normalisation of $(h^u k^v)$ to 
$I$ and suitable quotients of $h$ and $k$ to $L$. As in Case 2, the
quotients yield that $h$ and $k$ can be recovered from $L$ and $I[d]$.
Hence the result follows in this case also.
\end{proof}

We note two obvious improvements of the algorithm. 
\begin{items}
\item[(1)]
If there exists $l \in \{1, \ldots, n\}$ so that $l(I[d]) = 1$ 
for $l \leq d \leq n$, 
then we can improve the break in Step (2b) to: 'If $d \geq l$ then 
go back to (2)'.We can also replace the elements in $I$ so that 
$I[d] = g_d$ for $d \geq l$.
\item[(2)]
In Step (2d) we insert the normalisation of $w$ only if its
leading exponent is not equal to $b$. Further, if the leading 
exponent of the normalisation of $w$ equals either $a$ or $b$,
then only one left quotient needs to be added to $L$.
\end{items}

\subsection{Computing an igs}

The following algorithm takes a list $L$ of elements of $G$ and determines
an igs for the subgroup they generate. The algorithm is based on Lemma
\ref{igs}. 
\bigskip

{\bf IgsByGenerators}($L$):
\begin{items}
\item[(1)]
Initialise $I$ as a list of length $n$ with empty entries. 
\item[(2)]
While $L$ is not empty do:
\begin{items}
\item[(a)]
Take an element $g$ from $L$ and eliminate it in $L$.
\item[(b)]
Run {\bf AddGenToPIgs}($I$, $g$).
\item[(c)]
Let $N$ denote the list of changes to $I$ in (2b).
\item[(d)]
For $g$ in $N$ do:
\begin{items}
\item[(i)] If $r(g)$ is finite, then add $g^{r(g)}$ to $L$. 
\item[(ii)] For $h$ in $I$ with $h \neq g$ add $[g, h]$ to $L$.
\end{items}
\end{items}
\item[(3)] Return $I$.
\end{items}
\bigskip

The algorithm terminates, since the depths of the elements in $L$
increases in each step. The algorithm determines an igs for $\langle
L \rangle$, since it returns a list that generates $\langle L 
\rangle$ and satisfies the conditions of Lemma \ref{igs}.

\section{Computing the order and the index}

Suppose that a subgroup $U$ of $G$ is given by a set of generators.
Then $U$ and $[G:U]$ can both be read off from an igs of $U$.

\begin{lemma}
Let $u_1, \ldots, u_m$ be an igs for $U$, let $D = \{d(u_i) \mid
1 \leq i \leq m\}$ and let $\ol{D} = \{1, \ldots, n\} \setminus D$.
\begin{items}
\item[\rm (a)]
$|U| = r(u_1) \cdots r(u_m)$.
\item[\rm (b)]
$[G:U] = l(u_1) \cdots l(u_m) \cdot \prod_{d \in \ol{D}} r_d$.
\end{items}
\end{lemma} 

\section{Testing equality of subgroups}

Suppose that two subgroups $U$ and $V$ of $G$ are given. We would
like to have an effective test for $U = V$. We say that an igs
$u_1, \ldots, u_m$ is {\em canonical} if the normal forms
\[ u_i = g_1^{e_{i1}} \cdots g_n^{e_{in}} \]
satisfy that if $d(u_k) = d$ then 
$e_{id} \in \{0, \ldots, l(u_k)-1\}$ for $1 \leq i \leq m$.
It is not difficult to determine a canonical igs from an arbitrary
one by replacing $u_i$ by $u_i \cdot u_k^{-q}$ for all $i$ and $k$
where $e_{id} = l(u_k) q + r$ with $0 \leq r < l(u_k)$ is determined
by division with remainder.

\begin{lemma}
Two subgroups $U$ and $V$ are equal if and only if their canonical
igs coincide.
\end{lemma}

\def\cprime{$'$} \def\cprime{$'$}

\end{document}